\documentclass[12pt,reqno]{amsart}
\usepackage{amssymb,amsmath,enumerate,hyperref}
\usepackage{mathrsfs}
\usepackage{color}

\usepackage{mathrsfs}
\newtheorem{theorem}{Theorem}[section]
\newtheorem{proposition}[theorem]{Proposition}
\newtheorem{lemma}[theorem]{Lemma}

\theoremstyle{definition}

\newtheorem{example}[theorem]{Example}

\newcommand{\clb}{\mathscr{B}}

\newcommand{\clh}{\mathcal{H}}
\newcommand{\clk}{\mathcal{K}}

\newcommand{\clm}{\mathcal{M}}

\newcommand{\clp}{\mathcal{P}}
\newcommand{\clq}{\mathcal{Q}}

\newcommand{\bd}{\mathbb{D}}

\newcommand{\bt}{\mathbb{T}}

\begin{document}
	\title[Hyponormal Toeplitz operators]{Hyponormal Toeplitz operators  with operator valued normal symbols}

\author[Bala]{Neeru Bala}
\address{Department of Mathematics and Computing, IIT(ISM) Dhanbad, Jharkhand, 826004, India}
\email{neerusingh41@gmail.com, neerubala@iitism.ac.in}

\author[Pratima Pandey]{Pratima Pandey}
\address{Department of Mathematics and Computing, IIT(ISM) Dhanbad, Jharkhand, 826004, India}
\email{23dr0283@iitism.ac.in }



\begin{abstract}
	In this article, we first give a characterization of hyponormal Toeplitz operators with operator valued normal symbols. Latter, we examine when the class of hyponormal Toeplitz operators with operator valued symbols are analytic or normal.
\end{abstract}


\maketitle

	\section{Introduction}
	In this article, we denote a separable Hilbert space by $\clh$ and the space of all bounded linear operators on $\clh$ by $\clb(\clh)$. For $T\in\clb(\clh)$,  the null space and range space of $T$ are denoted by $\text{ ker}(T)$ and $\text{ran}(T)$, respectively. For a Hilbert space $\clh$, $L^{\infty}_{\clb(\clh)}(\bt)$, $H^{\infty}_{\clb(\clh)}(\bd)$ represent the Banach algebras of $\clb(\clh)$-valued essentially bounded Lebesgue measurable functions on $\bt$ and $\clb(\clh)$-valued bounded analytic functions on $\bd$, respectively. When $\clh=\mathbb{C}$, we write $L^{\infty}_{\clb(\clh)}(\bt)$ and $H^{\infty}_{\clb(\clh)} (\bd)$, by $L^{\infty}(\bt)$ and $H^{\infty}(\bd)$, respectively. For $\Phi\in L^{\infty}_{\clb(\clh)}(\bt)$, the {\it Toeplitz operator} $T_{\Phi}$ is defined by
	\begin{align*}
		T_{\Phi}f=P_+\Phi f\text{ for }f\in H^2_{\clh}(\bd),
	\end{align*}
	where $H^2_{\clh}(\bd)$ is the space of all $\clh$-valued analytic functions on $\bd$ and $P_+$ is the orthogonal projection from the space $L^2_{\clh}(\bt)$ of all $\clh$-valued Lebesgue square integrable functions on $\bt$ onto $H^2_{\clh}(\bd)$. When $\clh=\mathbb{C}$, we write $L^2_{\clh}(\bt)$ and $H^2_{\clh}(\bd)$ as $L^2(\bt)$ and $H^2(\bd)$, respectively. A Toeplitz operator $T_{\Phi}$ is called {\it analytic} if $\Phi\in H^{\infty}_{\clb(\clh)}(\bd).$
	
	Here our aim is to study hyponormal Toeplitz operators $T_{\Phi}$ for $\Phi\in L^{\infty}_{\clb(\clh)}(\bt)$. An operator $T\in\clb(\clh)$ is called {\it hyponormal}, if 
	\begin{align*}
		TT^*\leq T^*T,
	\end{align*}
	or the {\it self-commutator} $[T^*,T]=T^*T-TT^*\geq 0$.  The operator $T$ is called {\it subnormal} if there exists a Hilbert space $\clk\supseteq\clh$ and a normal operator $N\in\clb(\clk)$ such that $T=N|_{\clh}$. Note that the class of subnormal operators is contained in the class of hyponormal operators. One of the well knowm  questions posed by Halmos \cite{Halmos} related to Toeplitz operator is the following:
	\begin{center}
		"Is every subnormal Toeplitz operator $T_{\varphi}$ for $\varphi\in L^{\infty}(\bt)$ is normal or analytic?"\\
	\end{center}
	Cowen and Long \cite{Cowen Long} provided a negative answer to the above question, but it further motivated the study of Toeplitz operators which are normal or analytic. An interesting result was given by  Abrahamse \cite{Abrahamse1976} for bounded type functions. A function $\varphi\in L^{\infty}(\bt)$ is called {\it bounded type}, if $\varphi=\frac{\psi_1}{\psi_2}$ for some $\psi_1,\psi_2\in H^{\infty}(\bd)$.  In \cite{Abrahamse1976}, the author proved that if $T_{\varphi}$ is hyponormal, where $\varphi$ or $\bar{\varphi}$ is bounded type and $\text{ker}[T^*_{\varphi},T_{\varphi}]$ is invariant under $T_{\varphi}$, then $T_{\varphi}$ is normal or analytic. 
	
	For $\Phi\in L^{\infty}_{\clb(\clh)}(\bt)$, define
	\begin{align*}
		\Phi_+=P_+(\Phi)\text{ and }\Phi_-^*=(I-P_+)\Phi. 
	\end{align*} It is easy to see that $\Phi=\Phi_++\Phi^*_-.$ We say $\Phi$ is normal or hyponormal, if 
	\begin{align*}
		\Phi(z)^*\Phi(z)=\Phi(z)\Phi(z)^*\text{ or }	\Phi(z)^*\Phi(z)\geq\Phi(z)\Phi(z)^*\text{ for all }z\in\bd,
	\end{align*} respectively. The definition of bounded type functions is later extended by Curto et al. \cite{RE} to matrix valued functions in $L^{\infty}_{\clm_n}(\bt)$, where $\clm_n$ is the space of all $n\times n $ complex matrices. We say $\Phi=(\varphi_{ij})_{n\times n}\in L^{\infty}_{\clm_n}(\bt)$ is  bounded type, if every $\varphi_{ij}$ is bounded type for $1\leq i,j\leq n.$ Let $\Phi,\Psi\in H^{\infty}_{\clb(\clh)}(\bd)$. Then $\Phi$ and $\Psi$ are not {\it left coprime} if there exists a non-constant inner function $\Theta$ such that 
	$$\Phi(z)=\Theta(z)\Phi_1(z)\text{ and }\Psi(z)=\Theta(z)\Psi_1(z),$$
	for some $\Phi_1,\Psi_1\in H^{\infty}_{\clb(\clh)}(\bd)$. Curto et al. \cite{RE} extended the result of Abrahamse \cite{Abrahamse1976} to Toeplitz operators with matrix valued symbol and the following result.
	\begin{theorem}\cite{RE}
		Let $\Phi\in L^{\infty}_{\clm_n}(\bt)$ such that $\Phi$ and $\Phi^*$ are bounded type. Then
		$$\Phi_+=\Theta_1 A^*\text{ and }\Phi_-=\Theta_2 B^*,$$
		where $\Theta_i=\theta_i I_n$, $\theta_i\in H^{\infty}(\bd)$ and $A,B\in H^2_{\clm_n}(\bd)$. Assume that $A,B$ and $\Theta_2$ are left coprime. If
		\begin{enumerate}
			\item $T_{\Phi}$ is hyponormal.
			\item  $\text{ker}[T_{\Phi}^*,T_{\Phi}]$ is invariant under $T_{\Phi}$. 
		\end{enumerate}
		Then $T_{\Phi}$ is normal or analytic.
	\end{theorem}
		In this article, we extend the above result to Toeplitz operators $T_{\Phi}$, where $\Phi\in L^{\infty}_{\clb(\clh)}(\bt)$ is of the form
	\begin{align}\label{eqn main assumption}
		\Phi=\theta_1\Psi_1^*\text{ and }\Phi^*=\theta_2\Psi_2^*,
	\end{align}
	where $\theta_1,\theta_2\in H^{\infty}(\bd)$ are  inner functions and $\Psi_1,\Psi_2\in H^{\infty}_{\clb(\clh)}(\bd)$ are normal functions. We prove that in addition if $\text{ker}[T^*_{\Phi},T_{\Phi}]$ is invariant under $T_{\Phi}$ and $\theta_2$ is left coprime to $\Psi_1$ and $\Psi_2$, then $T_{\Phi}$ is normal or analytic. 
	
	Another interesting question related to hyponormal Toeplitz operators  is to characterize  symbol $\Phi$ for which $T_{\Phi}$ is hyponormal.
	For scalar valued $\varphi\in L^{\infty}(\bt)$, a complete characterization of hyponormal Toeplitz operators is given by Cowen \cite{Cowen}.
	\begin{theorem}\cite{Cowen}
		Let $\varphi\in L^{\infty}(\bt)$. Then $T_{\varphi}$ is hyponormal if and only if there exists a $k\in H^{\infty}(\bd)$ with $\|k\|_{\infty}\leq 1$ and $\varphi-k\bar{\varphi}\in H^{\infty}(\bd).$
	\end{theorem}
	The above result is latter extended by Gu et al. \cite{CG} to matrix valued Toeplitz operators and recently studied by Abhinand et al. \cite{Abhinand2025} for Toeplitz operators with operator valued symbol. In \cite{Abhinand2025} authors proved that if $\Phi\in L^{\infty}_{\clb(\clh)}(\bt)$ is hyponormal and there exist $K\in H^{\infty}_{\clb(\clh)}(\bt)$ such that $\|K\|_{\infty}\leq 1 $ and $\Phi-K\Phi^*\in H^{\infty}_{\clb(\clh)}(\bd)$, then $T_{\Phi}$ is hyponormal. In this article, we give a similar characterization for hyponormal Toeplitz operators with  normal operator valued function $\Phi\in L^{\infty}_{\clb(\clh)}(\bt)$. Explicitly saying, we proved that if $\Phi\in L^{\infty}_{\clb(\clh)}(\bt)$ is normal, then $T_{\Phi}$ is hyponormal if and only if there exists a $\Theta\in H^{\infty}_{\clb(\clh)}(\bd)$ with $$\|\Theta\|_{\infty}\leq 1\text{ and }\Phi-\Phi^*\Theta\in H^{\infty}_{\clb(\clh)}(\bd).$$
	
	Now, we mention a few notations which are frequently used in the subsequent sections. Let $\Theta\in H^{\infty}_{\clb(\clh)}(\bd)$. Then $\Theta$ is called an {\it inner function} if $\Theta(z)$ is an isometry for almost all $z\in\bt$ and called a {\it two sided inner function} if $\Theta(z)$ is a unitary for almost all $z\in\bt$. For an inner function $\Theta\in H^{\infty}_{\clb(\clh)}(\bd)$, the {\it model space} $\clq_{\Theta}$ is defined by
	$$\clq_{\Theta}=H^2_{\clh}(\bd)\ominus\Theta H^2_{\clh}(\bd).$$ For $\Phi\in L^{\infty}_{\clb(\clh)}(\bt)$, the {\it Hankel operator} $H_{\Phi}$ is defined by
	$$H_{\Phi}f=(I-P_+)\Phi f\text{ for }f\in H^2_{\clh}(\bd).$$
	For $\Phi,\Psi\in L^{\infty}_{\clb(\clh)}(\bt)$, we will be frequently using the relation
	\begin{align}\label{toeplitz hankel relation}
		T_{\Phi}T_{\Psi}=T_{\Phi\Psi}-H^*_{\Phi^*}H_{\Psi}.
	\end{align}
	We refer to \cite{VV} for more details about Toeplitz and Hankel operators.
	
	The article is divided into two section. In section two we prove our main result and study hyponormal Toeplitz operators for operators valued normal symbol.

	\section{Operator valued hyponormal Toeplitz operators}
	In this section, firstly we characterize hyponormal Toeplitz operators for an operator value normal symbol $\Phi\in L^{\infty}_{\clb(\clh)}(\bt)$. Note that normality of the function $\Phi$ will not imply normality of $T_{\Phi}$, simple examples are scalar inner functions and the corresponding Toeplitz operator. Thus it is interesting to study the class of hyponormal Toeplitz operators with normal operator valued symbol.
	For $\Theta\in H^{\infty}_{\clb(\clh)}(\bd)$, we use the notation  $\widetilde{\Theta}$ for the  $H^{\infty}_{\clb(\clh)}(\bd)$ function, which is defined by $\widetilde{\Theta}(z)=\Theta(\bar{z})^*$ for $z\in\bd$. 
		\begin{theorem}\label{theorem normal+hyponormal}
		Let  $\Phi \in L^\infty_{\clb(\clh)}(\bt)$ be  normal. Then  $T_\Phi$ is a hyponormal Toeplitz operator if and only if there exists a $\Theta \in H^\infty_{\clb(\clh)}(\bd)$ such that $\|\Theta\|_\infty\leq1$ and $\Phi-\Phi^* \widetilde\Theta \in H^\infty_{\clb(\clh)}(\bd).$
	\end{theorem}
	\begin{proof}
		Let $T_\Phi$ be a hyponormal operator. Then equation \eqref{toeplitz hankel relation} implies that
		\begin{align*}
			T_\Phi^* T_\Phi - T_\Phi T_\Phi^* &=
			T_{\Phi^* \Phi} - H_\Phi^* H_\Phi - T_{\Phi \Phi^*} + H_{\Phi^*}^* H_{\Phi^*} \\
			&=T_{\Phi^* \Phi - \Phi \Phi^*} - H_\Phi^* H_\Phi + H_{\Phi^*}^* H_{\Phi^*}.
		\end{align*}
		Since $\Phi$ is a normal operator and $T_{\Phi}$ is a hyponormal operator,  we have
		\begin{align*}
			T_\Phi^* T_\Phi - T_\Phi T_\Phi^*  =- H_\Phi^* H_\Phi + H_{\Phi^*}^* H_{\Phi^*}\geq 0.
		\end{align*}
		By  \cite[Theorem 1]{GR}, we get that $\text{ran}( H_{\Phi}^*)\subseteq \text{ran}(H_{\Phi^*}^* )$, and \cite[Theorem 7]{SR} implies that there exists a $\Theta \in H^\infty_{\clb(\clh)}(\bd)$ such that  $\|\Theta\|_\infty\leq1$ and
		$H_{\Phi}^* = H_{\Phi^*}^* T_\Theta.$ 
		Thus $H_{\Phi} = T_\Theta^* H_{\Phi^*}$ and \cite[Proposition~4]{SR} implies that
		\begin{align*}
			H_{\Phi} &= H_{\Phi^*} T_{\widetilde{\Theta}}= H_{\Phi^* \widetilde{\Theta}}.
		\end{align*}
		The above equation implies that  $H_{\Phi-\Phi^* \widetilde{\Theta}}=0$ and equivalently $\Phi - \Phi^* \widetilde{\Theta} \in H^\infty_{\clb(\clh)}(\bd)$.
		
		Conversely, assume that $\Phi - \Phi^* \widetilde{\Theta} \in H^\infty_{\clb(\clh)}(\bd).$ From \cite[Proposition 4]{SR}, we have
		\begin{align*}
			T^*_{\Phi}T_{\Phi}-T_{\Phi}T_{\Phi}^*=&H^*_{\Phi^*}H_{\Phi^*}-H^*_{\Phi}H_{\Phi}\\
			=&H^*_{\Phi^*}(I-T_{\Theta}T_{\Theta}^*)H_{\Phi^*}.
		\end{align*}
	Since $\|\Theta\|_{\infty}\leq 1$, we get that $T_{\Phi}$ is hyponormal.
	\end{proof}

Now, our next aim is to give a sufficient condition for a Toeplitz operator to be analytic or normal for operator valued symbol. Before proving our main result, we mention a few results which are used in our main result.
\begin{lemma}\label{containment lemma}
	Let $\theta\in H^{\infty}(\bd)$, $\Psi\in H^{\infty}_{\clb(\clk,\clh)}(\bd)$ be two inner functions with $	\Theta H^2_{\clh}(\bd)\subseteq\Psi H^2_{\clk}(\bd),$ where $\Theta=\theta I_{\clh}$. Then $\Psi^*(\theta I_{\clh})\in H^{\infty}_{\clb(\clh,\clk)}(\bd)$ is a two sided inner function.
\end{lemma}
\begin{proof}
	Let $	\Theta H^2_{\clh}(\bd)\subseteq\Psi H^2_{\clk}(\bd).$ By \cite[Page 727]{VV}, there exists an inner function $\Lambda\in H^{\infty}_{\clb(\clh,\clk)}(\bd)$ such that $\theta I_{\clh}=\Psi\Lambda$. Thus $\Lambda=\theta\Psi^*$ is an inner function. Also, observe  that
	$$\Lambda\Lambda^*=\theta\Psi^*\Psi\theta^*=I_{\clh}.$$
  		This proves our result.
	\end{proof}
\begin{proposition}\label{proposition analytic toeplitz}
	Let $\theta\in H^{\infty}(\bd)$ be an inner function, $\Psi\in H^{\infty}_{\clb(\clh)}(\bd)$ and $\Phi=\theta\Psi^*$. Then $\Phi_+^*\theta\in H^{\infty}_{\clb(\clh)}(\bd).$
\end{proposition}
\begin{proof}
		For $f\in H^2_{\clh}(\bd)$, we have $H_{\Phi^*}(\theta f)=0$ and as a result
			\begin{equation*}
			\theta H^2_\clh(\mathbb{D}) \subseteq \text{ker}(H_{\Phi^*})=\text{ker}(H_{\Phi_+^*}).
		\end{equation*}
		Since the kernel of a Hankel operator is an $M_z$ invariant subspace, by \cite[Theorem A 2.3]{VV} there exists a Hilbert space $\clk$ and an inner function $\Psi_1 \in H^\infty_{\clb(\clk,\clh)}(\mathbb{D})$ such that
		\begin{equation}\label{eqm1}
			\theta H^2_\clh(\mathbb{D}) \subseteq \text{ker}(H_{\varphi^*})=\text{ker}(H_{\varphi_+^*})=\Psi_1  H^2_\clk(\mathbb{D}). 
		\end{equation}
		By Lemma \ref{containment lemma}, we  know that $\Psi_1^*\theta$ is a two sided inner function.
		
			Let $\{e_j:j\in \mathbb{N}\}$ be an orthonormal basis for the Hilbert space $\clk$ and $$\Phi_+^*\theta=\underset{k\in\mathbb{Z}}{\sum}T_kz^k,$$ for $T_k\in \clb(\clh)$. From equation \eqref{eqm1}, we know that $\Phi_+^*\theta f\in H^2_{\clh}(\bd)$ for every $f\in H^2_{\clh}(\bd)$, which implies that 
		$  T_k(e_j)=0,$ for  $k<0,$ $j \in \mathbb{N}.$ Thus  $T_k=0, $ for $k<0$ and
		$$\Phi_+^*\theta=\sum_{k\in\mathbb{N}\cup\{0\}}T_kz^k\in H^{\infty}_{\clb(\clh)}(\bd).$$
		This proves our result.
	\end{proof}	
	The following result is an easy observation, but for completness of the article we give a short proof.
		\begin{proposition}\label{vector model space}
			Let $\theta\in H^{\infty}(\bd)$ and $\Psi\in H^{\infty}_{\clb(\clh)}(\bd)$ be two inner functions such that $\theta I_{\clh}$ and $\Psi$ are left coprime. Then
			\begin{enumerate}
				\item $\clq_{\theta I_{\clh}}=\overline{\{P_{\clq_{\theta I_{\clh}}}(\Psi g): g\in\clp_\clh\}},$ where $\clp_\clh$  is the   set of all $\clh$-valued polynomials.
				\item  $fe_j\in\clq_{\theta I_{\clh}}$ for $f\in \clq_{\theta}$ and $j\in\mathbb{N}$, where $\{e_i:i\in\mathbb{N}\}$ is an orthonormal basis for $\clh$.
			\end{enumerate}
		 	
		\end{proposition}
		\begin{proof}
			   Let $\Psi\in H^{\infty}_{\clb(\clh)}(\bd)$. Then $\Psi=\Psi_i\Psi_e$ (see \cite[Corollary 9]{NK}), where $\Psi_i\in H^{\infty}_{\clb(\tilde{\clh},\clh)}(\bd)$ is an inner function and $\Psi_e\in H^{\infty}_{\clb(\clh,\tilde{\clh})}(\bd)$ is an outer function for a Hilbert space $\tilde{\clh}$.
			 Observe that
			\begin{align*}
				\theta H^2_\clh(\mathbb{D}) \vee \overline{\{\Psi g: g\in\clp_\clh\}}&=\theta_1H^2_\clh(\mathbb{D}) \vee \overline{\{\Psi_i\Psi_e g: g\in\clp_\clh\}}\\
				&=\theta_1H^2_\clh(\mathbb{D}) \vee \Psi_iH^2_{\tilde{\clh}}(\mathbb{D}).  
			\end{align*}
			For the last equality, we refer to \cite[Page 20]{NK}. By our assumption, we know that $\theta$ and $\Psi$ are left coprime, which implies that $\theta$ and $\Psi_i$ are left coprime. Therefore
			\begin{align*}
				\theta H^2_\clh(\mathbb{D}) \vee \overline{\{\Psi g: g\in\clp_\clh \}}= H^2_\clh(\mathbb{D}).
			\end{align*}
			From the above equation, we get that $Q_{\theta I_\clh}\subseteq \overline{\{\Psi g: g\in\clp_\clh \}},$ and consequently $Q_{\theta I_\clh}=\overline{\{P_{\clq_{\theta I_{\clh}}}(\Psi g): g\in\clp_\clh\}}.$ 
			
			Let $f\in\clq_{\theta}$. For $j\in\mathbb{N}$, we have
			\begin{align*}
				\langle f(z)e_j, \theta(z)g(z)\rangle=\langle f(z),\theta(z)\langle e_j,g(z)\rangle\rangle=0\text{ for every }g\in H^2_{\clh}(\bd).
			\end{align*}
			From the above equation, we get that $fe_j\in\clq_{\theta I_{\clh}}$ for $j\in\mathbb{N}$ and $f\in\clq_{\theta}$.
		\end{proof}

Now, we prove our main result.

    \begin{theorem}
    	Let $\theta,\theta_1\in H^{\infty}(\bd)$ be inner functions and $\Psi,\Psi_1\in H^{\infty}_{\clb(\clh)}(\bd)$ be normal functions. Assume that
    	\begin{enumerate}
    		\item $\Phi=\theta\Psi^*$ and $\Phi^*=\theta_1\Psi_1^*$,
    		\item  $\theta_1$ is left coprime to $\Psi$ and $\Psi_1$,
            \item $T_\Phi$ is hyponormal,
            \item $\text{ker}[T_\Phi^*, T_\Phi]$ is invariant under $T_\Phi.$
        \end{enumerate}
        Then $T_\varphi$ is analytic or normal Toeplitz operator.
    \end{theorem}
    \begin{proof}
    	Let $\Phi=\theta\Psi^*$ and $\Phi^*=\theta_1\Psi_1^*$. 
    	By  Proposition \ref{proposition analytic toeplitz},we know that $\Phi_+^*\theta,\Phi_-^*\theta_1\in H^{\infty}_{\clb(\clh)}(\bd)$.
Consider $\Lambda_1=\Phi_+^*\theta $ and $\Lambda_2=\Phi_-^*\theta_1$. It is easy to see that 
\begin{align*}
	\Phi_+=\theta\Lambda_1^*,\Phi_-=\theta_1\Lambda_2^*.
\end{align*}
We claim that $\theta_1$, $\Lambda_1$ are coprime, as well as $\theta_1$ and $\Lambda_2$ are left coprime. 
On the contrary, assume that there exists an  inner function $\Delta_1\in H^{\infty}_{\clb(\clh)}(\bd)$, which is a common left  divisor of $\theta_1$ and $\Lambda_1$. Then 
        \begin{equation*}
            \Phi =\Phi_++\Phi_-^*=\theta(\Lambda_1^*+\theta^*\theta_1^*\Lambda_2).
        \end{equation*}
           Since $\theta$ is an inner function, we have 
        \begin{equation*}
          \Psi^*=\Lambda_1^*+\theta^*\Lambda_2\theta_1^*,
        \end{equation*} 
         which implies that $\Delta_1$ is a common left divisor of $\theta_1$ and $\Psi$. This is a contradiction to our assumption that $\theta_1$ and $\Psi$ are coprime. Hence $\theta_1$ and $\Lambda_1$ are left coprime. A similar argument for $\Phi^*=\theta_1\Psi_1^*$ implies that $\theta_1$ and $\Lambda_2$ are  coprime.

           Since $\Psi$ is normal, we have
           $$\Phi^* \Phi=\Psi \theta^*\theta \Psi^*=\Psi \Psi^*=\Psi^* \Psi=\theta \Psi^* \Psi \theta^*=\Phi \Phi^*.$$ Thus  $\Phi$ is normal. Since $T_\Phi$ is hyponormal, we have
       \begin{align*}
        [T_\Phi^*, T_\Phi]&= T_{\Phi^*\Phi-\Phi \Phi^*}-H_\Phi^*H_\Phi+H^*_{\Phi^*}H_{\Phi^*}\\
         &=H^*_{\Phi^*}H_{\Phi^*}-H_\Phi^*H_\Phi\\
         &=H_{\Phi_+^*}H_{\Phi_+^*}-H_{\Phi_-^*}^*H_{\Phi_-^*} \geq 0.
      \end{align*}
         From the above equation, we get that $\text{ker}H_{\Phi_+^*}\subseteq \text{ker}H_{\Phi_-^*},$ and as a result
      \begin{equation*}
         \theta H^2_\clh(\mathbb{D})\subseteq \theta_1 H^2_\clh(\mathbb{D}).
      \end{equation*}
      Using the fact that $\theta$ and $\theta_1$ are scalar inner functions along with \cite[Theorem 6]{SR}, we get that  there exists a scalar inner function $\theta_0$ such that $ \theta=\theta_1\theta_0.$ Thus
     \begin{align}\label{equation phi}
        \Phi_+=\theta_1\theta_0\Lambda_1^*,\text{ and }
         \Phi_-=\theta_1\Lambda_2^*.
     \end{align}
         If $\theta_1$ is a constant function, say $c$, then  $\Phi_-^*=\overline{c}\Lambda_2 \in H^\infty_{\clb(\clh)}(\mathbb{D}).$ Which is a contradiction to our assumption that $\Phi_-^*=(I-P_+)\Phi \notin H^\infty_{\clb(\clh)}(\mathbb{D}).$ Thus $\Phi_-=0,$ and as a result $T_\Phi$ is an analytic Toeplitz operator.

 On the other hand, assume that $\theta_1$ is a non-constant inner function. 
         Since $\Phi$ is normal and $T_{\Phi}$ is a hyponormal operator. By Theorem \ref{theorem normal+hyponormal}, there exists $K\in H^\infty_{\clb(\clh)}(\mathbb{D})$ such that $\Phi-\Phi^*K \in H^\infty_{\clb(\clh)}(\mathbb{D}),$ and consequently $\Phi_-^*-\Phi_+^*K \in H^\infty_{\clb(\clh)}(\mathbb{D}).$ Observe that
         \begin{align*}
         	[T_\Phi^*, T_\Phi] &=H_{\Phi_+^*}^*H_{\Phi_+^*}-H_{\Phi_+^*K}^*H_{\Phi_+^*K}\\
         	&=H_{\Phi_+^*}^*H_{\Phi_+^*}-H_{\Phi_+^*}^*T_{\tilde{K}}T_{\tilde{K}}^*H_{\Phi_+^*}\\
         	&=H_{\Phi_+^*}^*(I-T_{\tilde{K}}T_{\tilde{K}}^*)H_{\Phi_+^*}\\
         	&=((I-T_{\tilde{K}}T_{\tilde{K}}^*)^{\frac{1}{2}}H_{\Phi_+^*})^*(I-T_{\tilde{K}}T_{\tilde{K}}^*)^{\frac{1}{2}}H_{\Phi_+^*}.
         \end{align*}
         From the above equation, we get that 
      $$\text{ker}[T_\Phi^*, T_\Phi] \subseteq\text{ker}(I-T_{\tilde{K}}T_{\tilde{K}}^*)^{\frac{1}{2}}H_{\Phi_+^*}.$$
      
      We  claim that $\theta_0H^2_\clh(\mathbb{D})\subseteq \text{ker}[T_\Phi^*, T_\Phi].$ To prove our claim, first observe  that
     \begin{align*}
         [T_\Phi^*, T_\Phi]&=H_{\Phi_+^*}^*H_{\Phi_+^*}-H_{\Phi_-^*}^*H_{\Phi_-^*}\\
         &=H_{\Lambda_1\theta_0^* \theta_1^*}^*H_{\Lambda_1 \theta_0^* \theta_1^*}-H_{\Lambda_2 \theta_1^*}^*H_{\Lambda_2\theta_1^*},
     \end{align*}
     which implies that 
     \begin{equation}\label{eqn kernel cont 1}
        \theta_0\theta_1H^2_\clh(\mathbb{D}) \subseteq \text{ker}[T_\Phi^*, T_\Phi].
     \end{equation}
     By our assumptions, we know that $\text{ker}[T_\Phi^*, T_\Phi]$ is invariant under $T_{\Phi}$. Thus $T_{\Phi}(\theta_0\theta_1 H^2(\bd))\subseteq\text{ker}[T_\Phi^*, T_\Phi]$. For $f\in H^2_{\clh}(\bd)$, we have
      \begin{align}
     	T_\Phi(\theta_1\theta_0f)  =\Phi_+\theta_0 \theta_1f+P_{\theta_0\theta_1H^2_\clh(\mathbb{D})}(\theta_0 \Lambda_2 f)+P_{Q_{\theta_0 \theta_1 I_\clh}}(\theta_0\Lambda_2 f)    \end{align}
    which implies that $P_{Q_{\theta_0 \theta_1 I_\clh}}(\theta_0\Lambda_2 f)\in\text{ker}[T_\Phi^*, T_\Phi]$.  Using the fact that
$     	Q_{\theta_0\theta_1I_\clh}=Q_{\theta_0I_\clh}\oplus \theta_0Q_{\theta_1I_\clh},$
     we have
      \begin{align*}
     	P_{Q_{\theta_0\theta_1 I_\clh}}(\theta_0\Lambda_2 f)=P_{\theta_0 Q_{\theta_1 I_\clh}}(\theta_0\Lambda_2 f), \text{ for }f\in H^2_\clh(\mathbb{D}).
     \end{align*}
    Now, Proposition \ref{vector model space} implies that
       \begin{align}\label{equ kernel containment2}
         \theta_0Q_{\theta_1 I_\clh}=\overline{\{P_{\theta_0Q_{\theta_1 I_\clh}}(\theta_0\Lambda_2g): g\in \clp_\clh\}}\subseteq\text{ker}([T_\Phi^*,T_\Phi]).
     \end{align}
     From equations \eqref{eqn kernel cont 1} and \eqref{equ kernel containment2}, we have 
     \begin{align*}
   \theta_0H^2_\clh(\mathbb{D}) =   \theta_0\theta_1H^2_\clh(\mathbb{D}) \oplus\theta_0Q_{\theta_1 I_\clh} \subseteq \text{ker}[T_\varphi^*, T_\varphi].
     \end{align*}
     This proves our claim.
     
     For $f\in H^2_\clh(\mathbb{D})$, we have
       \begin{align*}
      (I-T_{\tilde{K}}T_{\tilde{K}}^*)^{\frac{1}{2}}H_{\Phi_+^*}\theta_0f
      &=(I-T_{\tilde{K}}T_{\tilde{K}}^*)^{\frac{1}{2}}H_{\Lambda_1 \theta_1^*\theta_0^*}\theta_0f\\ &
      = (I-T_{\tilde{K}}T_{\tilde{K}}^*)^{\frac{1}{2}}H_{\Lambda_1\theta_1^*}f \\
      &=0,
    \end{align*}
      which implies that $\text{ran}(H_{\Lambda_1\theta_1^*})\subseteq \text{ker}(I-T_{\tilde{K}}T_{\tilde{K}}^*).$ We already know that
    \begin{align*}
       \text{ker}(H_{\Lambda_1 \theta^*}^*)=\text{ker}(H_{\tilde{\Lambda}_1\tilde{\theta}_1^*})=\tilde{\theta}_1H^2_\clh(\mathbb{D}).
    \end{align*}
    Thus $
    Q_{\tilde{\theta}_1I_\clh}=\text{ran}(H_{\Lambda_1\theta_1^*})\subseteq ker(I-T_{\tilde{K}}T_{\tilde{K}}^*)$ and
   $T_{\tilde{K}}T_{\tilde{K}}^*f=f$ for $f\in  Q_{\tilde{\theta}_1I_\clh}$. We know that  $\|\tilde{K}^*\|_\infty=\|K\|_\infty\leq1,$ which implies that
     \begin{align*}
     \|T_{\tilde{K}}^*f\|\leq\|\tilde{K}^*f\|\leq\|f\|=\|T_{\tilde{K}}T_{\tilde{K}}^*f\|\leq\|{\tilde{K}}T_{\tilde{K}}^*f\|\leq\|T_{\tilde{K}}^*f\|\leq\|P({\tilde{K}}^*f)\|\leq \|\tilde{K}^*f\|\leq\|f\|,
     \end{align*}
     and consequently $\|P({\tilde{K}}^*f)\|= \|\tilde{K}^*f\|,$ or equivalently $\tilde{K}^*f\in H^2_\clh(\mathbb{D})),$ for every $f \in Q_{\tilde{\theta}_1I_\clh}.$ Thus
     \begin{align}\label{equ K}
      T_{\tilde{K}}T_{\tilde{K}}^*f=\tilde{K}\tilde{K}^*f=f,\text{ for }f\in Q_{\tilde{\theta}_1 I_\clh}.
\end{align}
We claim that $K(z)$ is a constant inner function, that is $K(z)=A_0$ for $z\in\bd$ and for some $A_0\in\clb(\clh)$. Assuming our claim we have
 \begin{align*}
	[T_\Phi^*, T_\Phi]&=H_{\Phi_+^*}^*(I-T_{\tilde{K}}T_{\tilde{K}}^*)H_{\Phi_+^*}\\
	&=H_{\Phi_+^*}^*(I-A_0^*A_0)H_{\Phi_+^*}\\
	&=0.
\end{align*}
Hence $T_\varphi$ is a normal operator. This proves our result.

    To prove our claim, we write $K(z)=\underset{k=0}{\overset{\infty}{\sum}}A_kz^k,$ where $A_k\in \clb(\clh)$ for  $k\in\mathbb{N}\cup\{0\}.$ By \cite[Lemma 3.4]{RE}, we can choose an outer function $f_0\in Q_{\tilde{\theta_1}},$ which is invertible in $H^\infty(\mathbb{D}).$ Let $\{e_j\}_{j\in\mathbb{N}}$ be an orthonormal basis for $\clh$. By Proposition \ref{vector model space}, we know that $f_0e_j\in\clq_{\tilde{\theta}_1 I_{\clh}}$ and consequently equation \eqref{equ K} implies that
    \begin{align*}
      \left( \sum_{k=0}^{\infty}A_k^*z^k\right) \left(\sum_{m=0}^{\infty}A_m\overline{z}^m\right)f_0e_j=f_0e_j\text{ for }j \in \mathbb{N} ,
    \end{align*}
  which is equivalent to say that
  \begin{align*}
     \left(\sum_{k=0}^{\infty} A_k^*A_k(e_j)+\sum_{n=1}^{\infty}\sum_{l=0}^{\infty}A_{n+l}^*A_l(e_j)z^n+\sum_{m=1}^{\infty}\sum_{p=0}^{\infty}A_p^*A_{p+m}(e_j)\overline{z}^m-e_j\right)f_0=0.
  \end{align*}
   Since $f_0$ is invertible in $H^\infty(\mathbb{D}),$ the above equation implies that
   $$ \sum_{k=0}^{\infty} A_k^*A_k(e_j)=e_j, ~\sum_{k=0}^{\infty}A_{n+k}^*A_k(e_j)=0, \text{ and } \sum_{k=0}^{\infty}A_{k}^*A_{k+n}(e_j)=0,\text{ for }j,n\in \mathbb{N}.$$
   Thus
  \begin{align*}
      \sum_{k=0}^{\infty} A_k^*A_k=I_\clh,~\sum_{k=0}^{\infty} A_{k+n}^*A_k=0,~\sum_{k=0}^{\infty} A_k^*A_{k+n}=0, \text{ for } n\in \mathbb{N},
  \end{align*}
 and $\tilde{K}(z)\tilde{K}^*(z)=I_\clh$ for $z\in\bd$. 
 
 Since $K(z)=\underset{k=0}{\overset{\infty}{\sum}}A_kz^k\in H_{\clb(\clh)}^\infty(\mathbb{D}),$ we have $K(z)e_j=\underset{k=0}{\overset{\infty}{\sum}}A_k(e_j)z^k\in H_{\clh}^2(\mathbb{D}),$ which implies that $\underset{k=0}{\overset{\infty}{\sum}}\|A_k(e_j)\|^2_\clh$ is finite  for $j\in \mathbb{N}.$ Recall that
 \begin{align*}
 	\underset{k=0}{\overset{\infty}{\sum}}|\langle A_k(e_j) , e_i\rangle|^2 \leq \underset{k=0}{\overset{\infty}{\sum}}\|A_k(e_j)\|^2,
 \end{align*}
and as a result we have
 \begin{align}\label{eqn kij}
 	K_{ij}(z)=\sum_{k=0}^{\infty}|\langle A_k(e_j) , e_i\rangle|z^k\in H^2(\mathbb{D}), \text{ for }i,j \in \mathbb{N}.
 \end{align}
 
We know that
 \begin{align*}
 	Q_{\tilde{\theta}_1I_\clh}\subseteq \text{ker}(I-T_{\tilde{K}}T_{\tilde{K}}^*)=\text{ker}(H_{K(\overline{z})}^*H_{K(\overline{z})})=\text{ker}(H_{K(\overline{z})}).
 \end{align*}
As $f_0e_j\in Q_{\tilde{\theta_1}I_\clh}\subseteq \text{ker}(H_{K(\overline{z})}),$ we have $K(\overline{z})f_0e_j\in  H^2_\clh(\mathbb{D}),$ for  $j\in \mathbb{N}.$
 Thus
 \begin{align*}
 	K_{ij}(\overline{z})f_0=\sum_{k=0}^{\infty}|\langle A_k(e_j) , e_i\rangle|\overline{z}^kf\in H^2(\mathbb{D})\text{ for }i,j \in \mathbb{N}.
 \end{align*}
 Since $f_0$ is invertible in $H^\infty(\mathbb{D}),$ we have 
  \begin{align}\label{eqn kij tilda}
      K_{ij}(\overline{z})=\sum_{k=0}^{\infty}|\langle A_k(e_j) , e_i\rangle|\overline{z}^k\in H^2(\mathbb{D}),\text{ for } i,j \in \mathbb{N}.
 \end{align}
     From equations \eqref{eqn kij} and \eqref{eqn kij tilda}, we get that $K_{ij}$ is a constant function for every $i,j\in\mathbb{N}.$ Hence $K(z)$ is a constant function. This proves our claim and completes the proof.
    \end{proof}
    	In the following example, we illustate that the conditions in above theorem are not necessary for the Toeplitz operator to be analytic or normal.
    \begin{example}
    	Let  $S\in\clb(\ell^2(\mathbb{N}))$ be the shift operator defined by
    	$$S(x_1,x_2,\ldots)=(0,x_1,x_2,\ldots).$$
    	Then
    	\(
    	\Theta(z) =zS
    	\in H^\infty_{\clb(l^2(\mathbb{N}))}    (\mathbb{D})\), and
    	\begin{align*}
    		T_\Theta^*T_\Theta-T_\Theta T_\Theta^*=P_{\ker(S^*)}\geq0.
    	\end{align*}
    	Thus $T_{\Theta}$ is a hyponormal operator.
Observe that $\text{ker}[T^*_{\Theta},T_{\Theta}]$ is not invariant under $T_{\Theta}$, but $T_{\Theta}$ is an analytic Toeplitz operator.
    \end{example}

\end{document}